\documentclass[mathpazo]{cicp}
\usepackage[margin=1in]{geometry} 
\usepackage{amsmath,amsthm,amssymb,amsfonts}
\usepackage{amsfonts}
\usepackage{subfigure}
\usepackage{graphicx}
\usepackage{float}
\usepackage{booktabs}
\usepackage{mathrsfs}
\usepackage{amsthm}
\usepackage{extarrows}
\usepackage{amssymb, amsmath, amscd, latexsym}
\usepackage{mathrsfs}
\usepackage{color}

\def\cG{\mathcal{G}}

\def\bx{\mathbf{x}}
\def\cL{\mathcal{L}}

\newcommand{\beq}{\begin{equation}}
\newcommand{\eeq}{\end{equation}}
\newcommand{\bea}{\begin{array}}
\newcommand{\eea}{\end{array}}

\newtheorem{scheme}{Scheme}

\begin{document}
\title{A Remark on the Invariant Energy Quadratization (IEQ) Method for Preserving the Original Energy Dissipation Laws}
\author[Zengyan Zhang, Yuezheng Gong and J. Zhao]{
Zengyan Zhang\affil{1} and Yuezheng Gong \affil{2,3} and 
Jia Zhao\affil{1}\comma \corrauth}
\address{
\affilnum{1}\ Department of Mathematics \& Statistics, Utah State University, Logan, UT, 84322, USA \\
\affilnum{2}\ Department of Mathematics, Nanjing University of Aeronautics and Astronautics, Nanjing 211106, China \\
\affilnum{3}\ Key Laboratory of Mathematical Modelling and High Performance Computing of Air Vehicles (NUAA), MIIT, Nanjing  211106, China
}
\email{ {\tt jia.zhao@usu.edu.} (J.~Zhao) \\ This is an invited manuscript for the Special Issue "Computational modeling and numerical analysis of the complex interfacial problems" in  journal Electronic Research Archive (ERA).}

\begin{abstract}
In this letter, we revisit the IEQ method and provide a new perspective on its ability to preserve the original energy dissipation laws. The invariant energy quadratization (IEQ) method has been widely used to design energy stable numerical schemes for phase-field or gradient flow models. Although there are many merits of the IEQ method, one major disadvantage is that the IEQ method usually respects a modified energy law, where the modified energy is expressed in the auxiliary variables. Still, the dissipation laws in terms of the original energy are not guaranteed. Using the widely-used Cahn-Hilliard equation as an example, we demonstrate that the Runge-Kutta IEQ method indeed can preserve the original energy dissipation laws for certain situations up to arbitrary high-order accuracy. Interested readers are highly encouraged to apply our idea to other phase-field equations or gradient flow models.
\end{abstract}

\ams{}
\keywords{Energy Stable; Cahn Hilliard Equation; Invariant Energy Quadratization (IEQ) Method}
\maketitle

\section{Introduction}
A wide variety of interfacial phenomena \cite{steinbach2009,Anderson1998,Elder2004,Guo2021,Liu2007} are driven by a certain dissipative mechanism \cite{onsager1931-1,onsager1931-2,yang2016}. As a powerful approach to describe the dissipative mechanism, a gradient flow model which respect the thermodynamical laws is commonly used. In general, consider a domain $\Omega$, and denote the state variable as $\phi$, the dissipative dynamics of $\phi$ takes the form of 
\begin{equation}\label{model}
\partial_t\phi({\bf x},t)=-\mathcal{G}\frac{\delta E}{\delta \phi}~~~ \text{in}~\Omega\times(0,T],
\end{equation}
where $E$ is a functional of $\phi$ known as the free energy, and $\mathcal{G}$ is a semi-positive operator, known as the mobility operator. The triplet $(\phi,\mathcal{G},E)$ uniquely determines the thermodynamically consistent gradient flow model. With proper initial values and boundary conditions, the dynamics of the gradient flow model \eqref{model} satisfies the following energy dissipation law,
\begin{equation}\label{law}
\frac{dE}{dt}=\Big( \frac{\delta E}{\delta \phi} , \frac{\partial \phi}{\partial t} \Big) =  -(\mathcal{G}\frac{\delta E}{\delta\phi},\frac{\delta E}{\delta\phi})\leq 0,
\end{equation}
where the inner product is defined by $\displaystyle(f,g)=\int_\Omega fg d\bx$, $\forall f,g\in L^2(\Omega)$.

Due to its broad applications in the literature, many approaches \cite{eyre1998,wang2010,guillen2013,shin2017,ju2018,shen2018, Zhao-2017-JCP} are proposed to develop numerical approximations such that the energy dissipation property \eqref{law} is able to be preserved in the discrete level. Among the existing numerical approaches, the invariant energy quadratization (IEQ) method \cite{Zhao-2017-JCP} has been extensively used to design numerical algorithms for a broad class of phase field models \cite{Zhao-2021-CICP-Chen}. Meanwhile, by combing it with the Runge-Kutta method, arbitrarily high-order schemes could be developed \cite{Zhao-2019-AML-Gong}. 

However, there is still one limitation of the IEQ method that has not been adequately addressed in the numerical analysis community. The numerical schemes based on the IEQ method mainly respect a modified energy law, where the modified energy is expressed with auxiliary variables. The original energy and the modified energy are equivalent with respect to the analytical solutions. But they are not necessarily equal with respect to the numerical solutions.  Whether the numerical schemes based on the IEQ method respect the original energy law is still unknown.

To remedy this issue, we introduced a relaxation technique in our early work \cite{ZhaoAML2021, JiangJCP2021} for solving phase-field equations with the IEQ method. By adding a relaxation parameter, we numerically penalize the difference between the modified energy and the original energy so that the numerical solutions of the phase-field equations will follow more closely to the original energy dissipation law. Nevertheless, this still leaves a gap to rigorously prove whether the numerical schemes based on the IEQ method respect the original energy law or not.

Inspired by some recent advances on designing high-order structure-preserving schemes  for Hamiltonian systems \cite{Gong2021, Tapley2021}, we revisit our early work on developing arbitrarily high-order numerical schemes for dissipative systems \cite{Zhao-2020-SISC, Zhao-2020-CPC, Zhao-2019-AML-Gong, Zhao-2020-JCP-2}. Eventually, we come up with a new perspective on the IEQ method. We conclude that specific Runge-Kutta-type numerical schemes derived by the IEQ method indeed can preserve the original energy dissipation law for certain situations up to arbitrary high-order accuracy.


\section{The IEQ method for the Cahn-Hilliard equation}
To better explain this new perspective, we restrict our presentation on the widely used Cahn-Hilliard equation in this paper.  Recall the widely used Cahn-Hilliard equation with a periodic boundary condition
\beq \label{eq:CH}
\begin{aligned}
&\partial_t \phi = M \Delta \mu, \\
&\mu = -\varepsilon \Delta \phi + \frac{1}{\varepsilon} (\phi^3 - \phi).
\end{aligned}
\eeq 
which can be written in an energy variation form of \eqref{model} with $\mathcal{G}=-M\Delta$.

This model has an energy dissipation law
\beq \label{eq:CH-energy-law}
\frac{d}{dt} E(\phi) = (M\Delta\mu,\mu)=-\int_\Omega M |\nabla \mu|^2 d\bx\leq0,
\eeq 
with the free energy $E(\phi)$ defined as
\beq \label{eq:original-energy}
E(\phi) =  \int_\Omega  \Big[ \frac{\varepsilon}{2}|\nabla \phi|^2 + \frac{1}{4\varepsilon}(\phi^2-1)^2 \Big] d\bx.
\eeq 

An energy stable scheme means the numerical solutions from the scheme will also respect the energy dissipation law of  \eqref{eq:CH-energy-law} in the discrete level. The IEQ method  \cite{Zhao-2017-JCP} is shown to be effective in guiding the design of energy stable schemes. The idea of the IEQ method is to reformulate the original PDE in \eqref{eq:CH} into an equivalent form, for which the energy stable schemes can be effectively designed. 

Specifically, we can introduce an auxiliary function 
\beq
q(\bx, t) := \phi^2 - 1 - C,
\eeq 
where $C$ is a constant (to be specified by the users) \cite{Zhao-2018-ANM}. With the auxiliary function $q(\bx, t)$, we can reformulate the CH equation \eqref{eq:CH} as 
\beq \label{eq:CH-EQ}
\begin{aligned}
& \partial_t \phi = M \Delta \mu, \\
& \mu  = - \varepsilon \Delta \phi + \frac{1}{\varepsilon} \phi (q + C), \\
& \partial_t q = 2\phi \partial_t \phi, \quad q(\bx, 0) = \phi^2(\bx, 0) - 1- C.
\end{aligned}
\eeq 
Here $q(\bx, 0) = \phi^2(\bx, 0) - 1- C$  is the consistent initial condition.
The reformulated model \eqref{eq:CH-EQ} satisfies a modified energy law
\beq \label{eq:CH-EQ-energy-law}
\frac{d }{dt}F(\phi, q) =- \int_\Omega M \Big|\nabla ( - \varepsilon \Delta \phi + \frac{1}{\varepsilon} \phi(q + C)) \Big|^2 d\bx\leq0,
\eeq 
with the modified energy $F(\phi,q)$ given as
\beq \label{eq:EQ-energy}
F(\phi, q) = \int_\Omega \Big[\frac{\varepsilon}{2}|\nabla \phi|^2 + \frac{C}{2\varepsilon}\phi^2 + \frac{1}{4\varepsilon}  \Big( q^2 - C^2 - 2C \Big) \Big] d\bx.
\eeq 

\begin{lemma}
\eqref{eq:CH} and \eqref{eq:CH-EQ} are equivalent.
\end{lemma}

This lemma can be easily shown. So the details are omitted.  Notice the fact that we can get 
\beq  \label{eq:consistent-condition}
q(\bx, t) - \Big[  (\phi(\bx,t))^2 -1 -C \Big] = 0
\eeq 
from \eqref{eq:CH-EQ}. Hence the original energy \eqref{eq:original-energy} and the modified energy \eqref{eq:EQ-energy} are equivalent, and the original energy law \eqref{eq:CH-energy-law} and the modified energy law \eqref{eq:CH-EQ-energy-law} are equivalent as well.
However, we emphasize that the consistent condition in \eqref{eq:consistent-condition} might not be satisfied numerically.

\section{Arbitrarily high-order numerical schemes}
To solve \eqref{eq:CH-EQ}, we revisit the IEQ-RK schemes in our previous work \cite{Zhao-2019-AML-Gong}. Consider the time domain $ t \in [0, T]$. We discretize it into equally distanced meshes, $0=t_0<t_1< \cdots < t_N=T$, with $t_i = i\Delta t$ and $\Delta t = \frac{T}{N}$. And we use $\phi^{n+1}$ to represent the numerical solutions of $\phi(\bx, t)$ at $t_{n+1}$. Similar notations apply to other variables as well.

\begin{scheme}[$s$-stage Runge-Kutta EQ scheme] \label{scheme1}
Let $a_{ij}$ and $b_i$ with $i,j=1,2,\cdots,s$ be real numbers (the Runge-Kutta coefficients). Use the the consistent initial condition $q^0 = (\phi^0)^2 - 1 -C$. Given $(\phi^n, q^n)$, we can calculate $(\phi^{n+1}, q^{n+1})$ through the following Runge-Kutta (RK) numerical scheme
\begin{eqnarray}
& \phi^{n+1} = \phi^n + \Delta t \displaystyle\sum_{i=1}^s b_i k_i^n, \label{scheme:phi}  \\
& q^{n+1} = q^n + \Delta t \displaystyle\sum_{i=1}^s b_i l_i^n, \label{scheme:q}
\end{eqnarray}
where the intermediate terms are calculated from
\begin{equation}
\begin{aligned}
& \phi_i^n = \phi^n + \Delta t \sum_{j=1}^s a_{ij} k_j^n, \\
& q_i^n = q^n + \Delta t \sum_{j=1}^s a_{ij} l_j^n, \\
& k_i^n = M\Delta\Big(- \varepsilon \Delta \phi_i^n + \frac{1}{\varepsilon} \phi_i^n (q_i^n + C)\Big), \\
& l_i^n = 2 \phi_i^n k_i^n, 
\end{aligned}
\end{equation}
with $i=1,2,\cdots,s$.
\end{scheme}

For simplicity of notations, we summarized the $s$-stage RK coefficients in the Butcher table form
$$
\begin{array}
{c|c}
\mathbf{c} & \mathbf{A} \\
\hline
&  \mathbf{b}^T\\
\end{array}
=
\begin{array}
{c|ccc}
c_1 & a_{11} & \cdots & a_{1s}  \\
\vdots & \vdots  & & \vdots \\
c_s & a_{s1} & \cdots & a_{ss} \\
\hline
& b_1 &  \cdots & b_s \\
\end{array},
$$
where $\mathbf{A} \in \mathbb{R}^{s,s}$, $\mathbf{b} \in \mathbb{R}^s$, and $\mathbf{c}=\mathbf{A} \mathbf{l}$ with $\mathbf{l}=(1,1,\cdots,1)^T \in \mathbb{R}^s$.

\begin{definition}[Symplectic Condition]
Define a symmetric matrix $\mathbf{S} \in \mathbb{R}^{s,s}$ as
$$
S_{ij} = b_i a_{ij} + b_j a_{ji} - b_i b_j, \quad i, j=1,2,\cdots s.
$$
The symplectic condition is defined as
\beq \label{eq:symplectic}
S_{ij} = 0, \quad  b_i \geq 0, \quad i, j=1,2,\cdots s.
\eeq 
\end{definition}

It is known from \cite{HairerBook} that the Gaussian collocation methods satisfy the symplectic condition in \eqref{eq:symplectic}. The RK coefficients based on the 2nd, 4th, and 6th order Gaussian collocation methods are summarized in Table \ref{tab-Guassian}.

\begin{table}[H]
\centering
\caption{Butcher tableaus of Gauss methods of 2, 4, and 6.}\label{tab-Guassian}

\begin{tabular}{ccc}
\hline
\begin{tabular}{c|c}
$\frac{1}{2}$ & $\frac{1}{2}$ \\
\hline
& 1 \\
\end{tabular}
&\begin{tabular}{c|cc}
$\frac{1}{2} - \frac{\sqrt{3}}{6}$ & $\frac{1}{4}$ & $\frac{1}{4} - \frac{ \sqrt{3}}{6}$  \\
$\frac{1}{2} + \frac{\sqrt{3}}{6}$ & $\frac{1}{4} + \frac{\sqrt{3}}{6}$  & $\frac{1}{4}$ \\
\hline
& $\frac{1}{2}$ & $\frac{1}{2}$ \\
\end{tabular}
&\begin{tabular}{c|ccc}
$\frac{1}{2} - \frac{\sqrt{15}}{10}$ & $\frac{5}{36}$ & $\frac{2}{9} - \frac{\sqrt{15}}{15}$ &  $\frac{5}{36} - \frac{\sqrt{15}}{30}$  \\
$\frac{1}{2}$ & $\frac{5}{36} + \frac{\sqrt{15}}{24}$ & $\frac{2}{9}$  & $\frac{5}{36} - \frac{\sqrt{15}}{24}$ \\
$\frac{1}{2} + \frac{\sqrt{15}}{10}$ & $\frac{5}{36} + \frac{\sqrt{15}}{30}$ & $ \frac{2}{9} + \frac{ \sqrt{15}}{15}$ & $\frac{5}{36}$  \\
\hline
& $ \frac{5}{18}$ & $\frac{4}{9}$ & $\frac{5}{18}$ \\
\end{tabular}\\
\hline
\end{tabular}
\end{table}
In other words, the set of RK coefficients that satisfy the symplectic condition in \eqref{eq:symplectic} is not empty. With these preparations, we are ready to present the main theorem in this letter.

\begin{theorem} \label{Theorem1}
Assume that the RK coefficients $a_{ij}, b_i$ satisfy the symplectic condition in \eqref{eq:symplectic}. Scheme \ref{scheme1} obeys the following energy dissipation law
\beq
E(\phi^{n+1}) - E(\phi^n)  = - \Delta t \sum_{i=1}^s b_i \Big\|\sqrt{M} \nabla \Big( -\varepsilon\Delta 
\phi_i^n+ \frac{1}{\varepsilon}\phi_i^n (q_i^n + C )\Big) \Big\|^2 \leq 0,
\eeq 
with $E(\phi)$ defined in \eqref{eq:original-energy}. That is the numerical solutions from Scheme \ref{scheme1} respect the original energy dissipation laws.
\end{theorem}

\begin{lemma}   \label{Lemma:1}
Assume that the RK coefficients $a_{ij}, b_i$ satisfy the symplectic condition in \eqref{eq:symplectic}. From Scheme \ref{scheme1}, we have
\beq
q^{n+1} =  (\phi^{n+1})^2 -1 -C.
\eeq 
\end{lemma}

\begin{proof}
We prove the Lemma \ref{Lemma:1} first.
First of all, we will show that 
$$q^{n+1} - q^n= (\phi^{n+1})^2  - (\phi^{n})^2.$$

From \eqref{scheme:q}, we have
\beq \label{eq:q1}
\begin{aligned}
q^{n+1} - q^n 
&= \Delta t \sum_{i=1}^s b_i l_i^n = 2\Delta t \sum_{i=1}^s b_i \phi_i^n k_i^n \\
& = 2\Delta t \sum_{i=1}^s b_i k_i^n \big(\phi^n+\Delta t\sum_{j=1}^s a_{ij} k_j^n\big) \\
& =  2\Delta t \sum_{i=1}^s b_i  k_i^n \phi^n+2 (\Delta t)^2 \sum_{i=1}^s \sum_{j=1}^s b_ia_{ij} k_i^n k_j^n \\
& = 2\Delta t \sum_{i=1}^s b_i  k_i^n \phi^n+ (\Delta t)^2 \sum_{i=1}^s \sum_{j=1}^s (b_i a_{ij} + b_j a_{ji}) k_i^n k_j^n \\
& = 2\Delta t \sum_{i=1}^s b_i  k_i^n \phi^n+ (\Delta t)^2 \sum_{i=1}^s \sum_{j=1}^s b_i b_j k_i^n k_j^n.
\end{aligned}
\eeq 
From \eqref{scheme:phi}, we have
\beq \label{eq:phi1}
\begin{aligned}
(\phi^{n+1})^2 - (\phi^n)^2 
& = \big(2 \phi^n + \Delta t \sum_{i=1}^s b_i k_i^n\big)(\Delta t \sum_{i=1}^s b_i k_i^n) \\
& = 2 \Delta t \sum_{i=1}^s b_i k_i^n \phi^n + (\Delta t)^2 \sum_{i=1}^s \sum_{j=1}^s b_i b_j k_i^n k_j^n.
\end{aligned}
\eeq 

Comparing the equations \eqref{eq:q1} and \eqref{eq:phi1}, we immediately find
\beq
q^{n+1} - q^n = (\phi^{n+1})^2 - (\phi^n)^2.
\eeq 

Also, notice the fact $q^0 = (\phi^0)^2 - 1- C$.  Therefore, by induction, we get
\beq
q^{n+1} =  (\phi^{n+1})^2 -1 -C.
\eeq 
\end{proof}

Now, we are ready to present the proof for Theorem \ref{Theorem1}.
\begin{proof}
Denote $\cL= -\varepsilon \Delta + \frac{1}{\varepsilon}C$. First, we will show that
\[E(\phi^{n+1})-E(\phi^n)=F(\phi^{n+1},q^{n+1})-F(\phi^n,q^n)\]
The original energy functional \eqref{eq:CH-energy-law} can be rewritten as the following quadratic form
\beq\label{qurad-energy}
E(\phi)=\frac{1}{2}(\phi,\cL\phi)-\frac{C}{2\varepsilon}(\phi,\phi)+\frac{1}{4\varepsilon}(\phi^2-1,\phi^2-1).
\eeq
Similarly, the modified energy functional \eqref{eq:EQ-energy} can be rewritten as follows
\beq\label{qurad-EQ-energy}
F(\phi,q)=\frac{1}{2}(\phi,\cL\phi)+\frac{1}{4\varepsilon}(q,q)-\frac{C^2+2C}{4\varepsilon}|\Omega|.
\eeq
From \eqref{qurad-energy}, \eqref{qurad-EQ-energy} and Lemma \ref{Lemma:1}, we have
\beq
\begin{aligned}
F(\phi^{n+1},q^{n+1})
& = \frac{1}{2}(\phi^{n+1},\cL\phi^{n+1})+\frac{1}{4\varepsilon}(q^{n+1},q^{n+1})-\frac{C^2+2C}{4\varepsilon}|\Omega|\\
& = \frac{1}{2}(\phi^{n+1},\cL\phi^{n+1})+\frac{1}{4\varepsilon}\Big(\big((\phi^{n+1})^2-1-C,(\phi^{n+1})^2-1-C\big)\Big)-\frac{C^2+2C}{4\varepsilon}|\Omega|\\
& = \frac{1}{2}(\phi^{n+1},\cL\phi^{n+1})+\frac{1}{4\varepsilon}\big((\phi^{n+1})^2-1,(\phi^{n+1})^2-1\big)-\frac{C}{2\varepsilon}(\phi^{n+1},\phi^{n+1})\\
& = E(\phi^{n+1}).
\end{aligned}
\eeq
Therefore, we can get
\beq
\begin{aligned}
E(\phi^{n+1})-E(\phi^n)
& = F(\phi^{n+1},q^{n+1})-F(\phi^n,q^n)\\
& = \frac{1}{2}(\phi^{n+1}, \cL \phi^{n+1})  + \frac{1}{4\varepsilon}(q^{n+1}, q^{n+1})   - \Big[ \frac{1}{2}(\phi^n, \cL\phi^n)  + \frac{1}{4\varepsilon}(q^n, q^n) \Big].
\end{aligned}
\eeq
Next, from \eqref{scheme:phi}, we have
\beq \label{eq:e-1}
\begin{aligned}
& \frac{1}{2}(\phi^{n+1}, \cL \phi^{n+1}) - \frac{1}{2}(\phi^n, \cL\phi^n)  \\
= &\frac{1}{2}(\phi^n+\Delta t\sum_{i=1}^s b_i k_i^n,\cL\phi^n+\Delta t\sum_{i=1}^s b_i \cL k_i^n)-\frac{1}{2}(\phi^n,\cL\phi^n)\\
= &   \Delta t \sum_{i=1}^s b_i (k_i^n, \cL \phi^n) + \frac{1}{2} (\Delta t)^2 \sum_{i=1}^s \sum_{j=1}^s b_i b_j (k_i^n, \cL k_j^n) \\
= & \Delta t\sum_{i=1}^s b_i(k_i^n, \cL\phi_i^n-\Delta t\sum_{j=1}^s a_{ij}\cL k_j^n)+ \frac{1}{2} (\Delta t)^2 \sum_{i=1}^s \sum_{j=1}^s b_i b_j (k_i^n, \cL k_j^n)\\
= & \Delta t \sum_{i=1}^s b_i (k_i^n , \cL \phi_i^n)-(\Delta t)^2\sum_{i=1}^s\sum_{j=1}^s b_i a_{ij} (k_i^n,\cL k_j^n)+ \frac{1}{2} (\Delta t)^2 \sum_{i=1}^s \sum_{j=1}^s b_i b_j (k_i^n, \cL k_j^n)\\ 
= & \Delta t \sum_{i=1}^s b_i (k_i^n , \cL \phi_i^n) - \frac{1}{2}(\Delta t)^2 \sum_{i=1}^s \sum_{j=1}^s \Big[ b_i a_{ij} + b_j a_{ji} - b_i b_j \Big] (k_i^n, \cL k_j^n) \\
= & \Delta t \sum_{i=1}^s b_i (k_i^n , \cL \phi_i^n).
\end{aligned}
\eeq 
Meanwhile, from \eqref{scheme:q}, we have
\beq \label{eq:e-2}
\begin{aligned}
& \frac{1}{2}(q^{n+1}, q^{n+1}) - \frac{1}{2}(q^n, q^n)  \\
= & \frac{1}{2}(q^n+\Delta t\sum_{i=1}^s b_i l_i^n,q^n+\Delta t\sum_{i=1}^s b_i l_i^n)-\frac{1}{2}(q^n,q^n)\\
= &   \Delta t \sum_{i=1}^s b_i (l_i^n, q^n) + \frac{1}{2} (\Delta t)^2 \sum_{i=1}^s \sum_{j=1}^s b_i b_j (l_i^n, l_j^n) \\
= &  \Delta t \sum_{i=1}^s b_i (l_i^n, q_i^n-\Delta t\sum_{j=1}^sa_{ij} l_j^n)+ \frac{1}{2} (\Delta t)^2 \sum_{i=1}^s \sum_{j=1}^s b_i b_j (l_i^n, l_j^n)\\
= & \Delta t \sum_{i=1}^s b_i (l_i^n , q_i^n) - \frac{1}{2}(\Delta t)^2 \sum_{i=1}^s \sum_{j=1}^s \Big[ b_i a_{ij} + b_j a_{ji} - b_i b_j \Big] (l_i^n, l_j^n) \\
= & \Delta t \sum_{i=1}^s b_i (l_i^n , q_i^n)= \Delta t \sum_{i=1}^s b_i (k_i^n, 2\phi_i^n q_i^n).
\end{aligned}
\eeq 

So, from \eqref{eq:e-1} and \eqref{eq:e-2}, we conclude
\beq
\begin{aligned}
& \frac{1}{2}(\phi^{n+1}, \cL \phi^{n+1})  + \frac{1}{4\varepsilon}(q^{n+1}, q^{n+1})   - \Big[ \frac{1}{2}(\phi^n, \cL\phi^n)  + \frac{1}{4\varepsilon}(q^n, q^n) \Big] \\
= & \Delta t \sum_{i=1}^s b_i (k_i^n , \cL \phi_i^n) +\frac{\Delta t}{2\varepsilon} \sum_{i=1}^s b_i (k_i^n, 2\phi_i^n q_i^n)\\
= & \Delta t \sum_{i=1}^s b_i(k_i^n, \cL \phi_i^n+\frac{1}{\varepsilon}\phi_i^n q_i^n)\\
= & \Delta t \sum_{i=1}^s b_i (M\Delta(\cL
\phi_i^n+\frac{1}{\varepsilon}\phi_i^n q_i^n) , \cL \phi_i^n+\frac{1}{\varepsilon}\phi_i^n q_i^n).
\end{aligned}
\eeq 

Noticing the semi-definite property of $\cG=-M\Delta$ and $b_i\geq 0$, $\forall i$, this finally leads to
\beq
\begin{aligned}
E(\phi^{n+1}) - E(\phi^n)  
&= \Delta t \sum_{i=1}^s b_i (M\Delta(\cL
\phi_i^n+\frac{1}{\varepsilon}\phi_i^n q_i^n) , \cL \phi_i^n+\frac{1}{\varepsilon}\phi_i^n q_i^n) \\
& = - \Delta t \sum_{i=1}^s b_i \Big\|\sqrt{M} \nabla \Big( -\varepsilon\Delta 
\phi_i^n+ \frac{1}{\varepsilon}\phi_i^n (q_i^n + C )\Big) \Big\|^2 \\
&\leq 0,
\end{aligned}
\eeq 
where $E(\phi)$ is defined in \eqref{eq:original-energy}.
\end{proof}

\section{Connections with the classical implicit schemes}
The previous section makes it clear that the IEQ method can be used to derive arbitrarily high-order accurate numerical schemes for the Cahn-Hilliard equation that preserve the original energy dissipation laws. 

In particular, when $s=1$, we have the RK coefficients in Table \ref{tab-Guassian}. Scheme \ref{scheme1} is reduced as
\beq \label{eq:2nd}
\begin{aligned}
& \phi^{n+1} = \phi^n + \Delta t k_1^n, \\
& q^{n+1} = q^n + \Delta t l_1^n, \\
&\phi_1^n = \phi^n + \frac{1}{2} \Delta t k_1^n, \\
&q_1^n = q^n + \frac{1}{2} \Delta t  l_1^n, \\
&k_1^n = M \Delta ( - \varepsilon \Delta \phi_1^n + \frac{1}{\varepsilon}\phi_1^n (q_1^n + C)), \\
&l_1^n = 2 \phi_1^n k_1^n,
\end{aligned}
\eeq 
with the consistent initial condition $q^0 = (\phi^0)^2 - 1 - C$.

\begin{theorem}
The numerical scheme in \eqref{eq:2nd} is equivalent to the implicit scheme
\beq \label{eq:implicit-2nd}
\begin{aligned}
& \frac{\phi^{n+1} - \phi^n}{\Delta t} = M \Delta \mu^{n+\frac{1}{2}}, \\
& \mu^{n+\frac{1}{2}} = - \varepsilon \Delta \frac{\phi^{n+1} + \phi^n}{2} + \frac{1}{2 \varepsilon}(\phi^{n+1} + \phi^n) \Big[ \frac{1}{2} ((\phi^{n+1})^2 + (\phi^n)^2)  -1 \Big].
\end{aligned}
\eeq 
\end{theorem}

\begin{proof}
From \eqref{eq:2nd}, we have
\beq \label{eq:sch-part1}
\frac{\phi^{n+1} - \phi^n}{\Delta t} = k_1^n, 
\eeq 
so that 
$$\phi_1^n = \phi^n + \frac{1}{2}\Delta t \frac{\phi^{n+1} - \phi^n}{\Delta t} = \frac{1}{2} (\phi^{n+1} + \phi^n).$$
Similarly, we have $q_1^n = \frac{1}{2}(q^n + q^{n+1})$.

Based on Theorem \ref{Theorem1}, we have $q^{n+1} = (\phi^{n+1})^2 - 1 - C$. This leads to
\beq \label{eq:q1n}
q_1^n = \frac{1}{2}( (\phi^{n+1})^2 + (\phi^n)^2) - 1 -C.
\eeq 
Substituting \eqref{eq:q1n} back into the expression for $k_1^n$, we have
\beq \label{eq:sch-part2}
k_1^n = M \Delta \Big[  - \varepsilon \Delta \frac{\phi^{n+1} + \phi^n}{2} + \frac{1}{2\varepsilon} (\phi^{n+1} + \phi^n) ( \frac{1}{2}( (\phi^{n+1})^2 + (\phi^n)^2) - 1) \Big].
\eeq

Combing \eqref{eq:sch-part1} and \eqref{eq:sch-part2}, we arrive at the implicit scheme \eqref{eq:implicit-2nd}.
\end{proof}

\section{Concluding remarks}
In this letter, we have revisited the IEQ method for solving the Cahn-Hilliard equation. In particular, we point out that the EQ-RK schemes with specified Runge-Kutta coefficients can preserve the energy dissipation laws with respect to the original energy expression. This partially addresses the opening question of whether the numerical schemes based on the IEQ method respect the original energy laws. It also sheds light on further exploring the IEQ method for solving thermodynamically consistent models.

\section*{Acknowledgments}
Z. Zhang and J. Zhao would like to acknowledge the support from National Science Foundation, United States, with grant NSF-DMS-2111479.

\bibliographystyle{plain}

\end{document}